\documentclass[a4paper,12pt,draft]{amsart}
\usepackage[margin=30mm]{geometry}
\usepackage{amssymb, amsmath, amsthm, amscd}

\usepackage[T1]{fontenc}
\usepackage{eucal,mathrsfs,dsfont}
\usepackage{color}

\definecolor{mno}{rgb}{0.5,0.1,0.5}

\newcommand{\R}{\mathds R}
\newcommand{\Rd}{{\mathds R^d}}
\newcommand{\Pp}{\mathds P}
\newcommand{\Ee}{\mathds E}
\newcommand{\I}{\mathds 1}
\newcommand{\Z}{\mathds Z}
\newcommand{\var}{\textmd{Var}}
\newcommand{\Bb}{\mathscr{B}}

\newtheorem{theorem}{Theorem}[section]
\newtheorem{lemma}[theorem]{Lemma}
\newtheorem{proposition}[theorem]{Proposition}
\newtheorem{corollary}[theorem]{Corollary}

\theoremstyle{definition}

\newtheorem{remark}[theorem]{Remark}
\newtheorem*{acknowledgement}{Acknowledgement}

\begin{document}

\title[Coupling Property of L\'{e}vy Processes]{\bfseries On the Coupling Property of L\'{e}vy Processes}

\author{Ren\'{e} L.\ Schilling \qquad Jian Wang}
\thanks{\emph{R.\ Schilling:} TU Dresden, Institut f\"{u}r Mathematische Stochastik, 01062 Dresden, Germany. \texttt{rene.schilling@tu-dresden.de}}
\thanks{\emph{J.\ Wang:} School of Mathematics and Computer Science, Fujian Normal
University, 350007, Fuzhou, P.R. China \emph{and} TU Dresden,
Institut f\"{u}r Mathematische Stochastik, 01062 Dresden, Germany.
\texttt{jianwang@fjnu.edu.cn}}

\date{}

\maketitle

\begin{abstract}
We give necessary and sufficient conditions guaranteeing that the coupling for
 L\'evy processes (with non-degenerate jump part) is successful.
 Our method relies on explicit formulae for the transition semigroup of
 a compound Poisson process and earlier results by Mineka and Lindvall-Rogers on couplings of random walks.
 In particular, we obtain that a L\'{e}vy process admits a successful coupling,
 if it is a strong Feller process or if the L\'evy (jump) measure has an absolutely continuous component.

\medskip

\medskip

\end{abstract}

\medskip

\begin{center}{\large \bfseries Sur la propri\'{e}t\'{e} de couplage des
processus de L\'{e}vy}\end{center}

\begin{abstract}[R\'{e}sum\'{e}]
Nous donnons les conditions n\'{e}cessaires et suffisantes
garantissant le succ\`{e}s de couplage des processus de L\'{e}vy
(avec une jump-part non-d\'{e}g\'{e}n\'{e}r\'{e}e). Notre
m\'{e}thode est bas\'{e}e sur les formules explicites pour le
semigroupe de transition d'un processus de Poisson compos\'{e}, et
les r\'{e}sultats de Mineka et Lindvall-Rogers sur le couplage de
marche al\'{e}atoire. En particulier, nous montrons qu'un processus
de L\'{e}vy admet un couplage r\'{e}ussi, s'il est un processus
fortement fellerien ou si la mesure de L\'{e}vy (saut) poss\`{e}de
une composante absolument continue.

\medskip

\noindent\textbf{Mots-cl\'{e}s:} Propri\'{e}t\'{e} de couplage;
processus de L\'{e}vy; processus de Poisson compos\'{e}; marche
al\'{e}atoire; couplage de  Mineka; Propri\'{e}t\'{e} forte de
Feller.
\end{abstract}

\section{Introduction and Main Results}\label{section1}
The coupling method is a powerful tool in the study of Markov
processes and interacting particle systems. There are some
comprehensive books on this topic now, see e.g.\
\cite{Li,T,Chen,Wang}. Let $(X_t)_{t\ge0}$ be a Markov process on $\R^d$ with transition probability function $\{P_t(x,\cdot)\}_{t\ge 0, x\in\R^d}$. An $\R^{2d}$-valued process $(X'_t,X''_t)_{t\ge0}$ is
called \emph{a coupling of the Markov process} $(X_t)_{t\ge0}$, if
both $(X'_t)_{t\ge0}$ and $(X''_t)_{t\ge0}$ are Markov processes which have the same
transition functions $P_t(x,\cdot)$ but possibly different initial
distributions. In this case, $(X'_t)_{t\ge0}$ and $(X''_t)_{t\ge0}$ are called the \emph{marginal
processes} of the coupling process; the coupling time is
defined by $T:=\inf\{t\ge0: X'_t=X''_t\}$. The coupling $(X'_t,X''_t)_{t\ge0}$ is
said to be \emph{successful} if $T$ is a.s.\ finite. A Markov process $(X_t)_{t\ge0}$
admits a \emph{successful coupling} (also: enjoys the \emph{coupling property}) if
for any two initial distributions $\mu_1$ and $\mu_2$, there exists
a successful coupling with marginal processes possessing the same transition probability functions $P_t(x,\cdot)$ and starting from $\mu_1$
and $\mu_2$, respectively. It is known, see \cite{Li, T1}, that the
coupling property is equivalent to the statement that
\begin{equation}\label{prex1}
    \lim_{t\rightarrow\infty}\|\mu_1P_t-\mu_2P_t\|_{\var}=0\qquad\textrm{ for } \mu_1\text{\ and\ }\mu_2\in \mathscr{P}(\R^d).
\end{equation}
As usual, $\mu P(A) = \int P(x,A)\,\mu(dx)$ is the left action of
the semigroup and $\|\cdot\|_{\var}$ stands for the total variation
norm. If a Markov process admits a successful coupling, then it also
has the Liouville property, i.e.\ every bounded harmonic function is
constant; in this context a function $f$ is harmonic, if $Lf=0$
where $L$ is the generator of the Markov process. See \cite{CG,CW}
and references therein for this result and more details on the
coupling property.

The aim of this paper is to study the coupling property of L\'{e}vy processes by using explicit conditions on L\'{e}vy measures. Our work is mainly motivated by the recent paper \cite{W1}, which contains some interesting results on the coupling property of Ornstein-Uhlenbeck processes; the paper \cite{W1} uses mainly the conditional Girsanov theorem on Poisson space and assumes that the corresponding L\'{e}vy measure has a non-trivial absolutely continuous part. Our technique here is completely different from F.-Y.\ Wang's paper \cite{W1}. We use an explicit expression of the compound Poisson semigroup and combine this with the Mineka- and Lindvall-Rogers-couplings for random walks, see \cite{LR}.

A L\'evy process $(X_t)_{t\geq 0}$ on $\Rd$ is a stochastic process with stationary and independent increments and c\`adl\`ag (right continuous with finite left limits) paths. It is well known that $X_t$ is a (strong) Markov process whose infinitesimal generator is, for $f\in C_b(\Rd)$, of the form
$$
    Lf(x)=  \frac 12\sum_{i,j=1}^d q_{i,j} \frac{\partial^2 f(x)}{\partial x_i \partial x_j} + \int \bigg[f(x+z)-f(x)- \frac{z\cdot\nabla f(x)}{1+|z|^2}\bigg]\nu(dz)+b\cdot\nabla f(x),
$$
where $Q=(q_{i,j})_{i,j=1}^d$ is a positive semi-definite matrix,
$b\in\Rd$ is the drift vector and $\nu$ is the L\'evy or jump
measure; the L\'{e}vy measure $\nu$ is a $\sigma$-finite measure on
$(\R^d,\Bb(\R^d))$ such that $\int(1\wedge |z|^2)\nu(dz)<\infty$.
Note that the L\'evy triplet $(b,Q,\nu)$ characterizes, up to
indistinguishability, the process $(X_t)_{t\geq 0}$ uniquely.
Our standard reference for L\'{e}vy processes
is the monograph \cite{SA}.
 We write $P_t(x,A) = P_t(A-x)$,
$A\in\Bb(\Rd)$, for the transition probability of $X_t$.

Let $\mu$ and $\nu$ be two bounded measures on $(\R^d,\Bb(\R^d))$.
We define
$\mu\wedge\nu:=\mu-(\mu-\nu)^+$, where $(\mu-\nu)^{\pm}$ is the
Jordan-Hahn decomposition of the signed measure $\mu-\nu$. In
particular, $\mu\wedge\nu=\nu\wedge\mu$ and $\mu\wedge
\nu\,(\R^d)=\frac{1}{2}\big[\mu(\R^d)+\nu(\R^d)-\|\mu-\nu\|_{\var}\big],$ cf.\ \cite{Chen}.
 We can now state our main result.

\begin{theorem}\label{th2}
Let $(X_t)_{t\geq 0}$ be a $d$-dimensional L\'evy process with
L\'evy triplet $(b,Q,\nu)$. For every $\varepsilon>0$, define
${\nu}_\varepsilon$ by
\begin{equation}\label{eq:th2}
    {\nu}_\varepsilon(B)
    =
    \begin{cases}
        \qquad\nu(B),\qquad& \nu(\R^d)<\infty;\\
        \nu\{z\in B: |z|\ge \varepsilon\},& \nu(\R^d)=\infty.
    \end{cases}
\end{equation}
Assume that there exist $\varepsilon, \delta>0$ such that
\begin{equation}\label{th22}
    \inf_{x\in\R^d,|x|\le \delta}\nu_\varepsilon\wedge (\delta_x*\nu_\varepsilon)(\R^d)>0.
\end{equation}
Then, there exists a constant $C=C(\varepsilon,\delta,\nu)>0 $ such that for all $x, y\in\R^d$ and $t>0$,
\begin{equation*}\label{th21}
    \|P_t(x,\cdot)-P_t(y,\cdot)\|_{\var}
    \le
    \frac{C(1+|x-y|)}{\sqrt{t}}\wedge2.
\end{equation*}
In particular, the L\'{e}vy process $X_t$ admits a successful coupling.
\end{theorem}

\begin{remark}\label{th22re}
Condition \eqref{th22} guarantees that
$$
    \|P_t(x,\cdot)-P_t(y,\cdot)\|_{\var}
    =\mathsf{O}(t^{-1/2})\qquad\text{\ as\ }\,\,t\rightarrow\infty .
$$
holds locally uniformly for all $x,y\in\R^d$. This order of
convergence is known to be optimal for compound Poisson processes,
see \cite[Remark 3.1]{W1}. In \cite{W1} it is pointed out that a
pure jump L\'evy process admits a successful coupling only if the
L\'{e}vy measure has a non-discrete support, in order to make the
process more active. Condition \eqref{th22} is one possibility to
guarantees sufficient jump activity; intuitively it will hold if for
sufficiently small values of $\varepsilon, \delta>0$ and all
$x\in\R^d$ with $|x|\le \delta$ we have
$x+\text{supp}(\nu_\varepsilon)\cap\text{supp}(\nu_\varepsilon)\neq\emptyset$;
here $\text{supp}(\nu_\varepsilon)$ is the support of the measure
$\nu_\varepsilon$.

In order to see that \eqref{th22} is sharp, we consider an
one-dimensional compound Poisson process with L\'{e}vy measure $\nu$
supported on $\Z$. Then, for any $\delta\in(0,1)$ and $x\in\R^d$
with $|x|\le \delta$, $\nu\wedge (\delta_x*\nu)(\Z)=0$. On the other
hand, all functions satisfying $f(x+n)=f(x)$ for $x\in\R^d$ and
$n\in \Z$ are harmonic. By \cite{CG,CW}, this process cannot have
the coupling property.
\end{remark}

Theorem 3.1 of \cite{W1} establishes the coupling property for L\'{e}vy processes whose jump measure  $\nu$ has a non-trivial absolutely continuous part. It seems to us that this condition is not
directly comparable with \eqref{th22}. In fact, based on the Lindvall-Rogers `zero-two law' for random walks \cite[Propsotion 1]{LR}, we give in Section \ref{section4} a necessary and sufficient condition guaranteeing that a L\'evy process has the coupling property. In this section we will also find the connection between \eqref{th22} and the existence of a non-trivial absolutely
continuous component of the L\'{e}vy measure. In particular, we obtain some extensions of Theorem \ref{th2} and \cite[Theorem 3.1]{W1}.

\medskip

Once we know that a L\'evy process admits the coupling property, many interesting new questions arise which are, however, beyond the scope of the present paper. For example, it would be interesting to construct explicitly the corresponding successful Markov coupling
process and to determine its infinitesimal operator. There are a number of applications
of optimal Markov processes and operators; we refer to \cite{Chen,Wang} for background material and a more detailed account on diffusions and $q$-processes. We will discuss those topics for L\'{e}vy processes in a forthcoming paper \cite{BSW}.

\section{The Coupling Property of Compound Poisson Processes}\label{section2}
In this section, we consider the coupling property of compound Poisson processes. Let $(L_t)_{t\geq 0}$ be a compound Poisson process on $\R^d$ such that $L_0 = x$ and with L\'{e}vy measure $\nu$. Then, $L_t$ can be written as
\begin{equation*}\label{com1}
    L_t=x+\sum_{i=1}^{N_t}\xi_i, \qquad t\ge0,
\end{equation*}
where $N_t$ is the Poisson process with rate $\lambda:=\nu(\R^d)$ and $(\xi_i)_{i\ge1}$ is a sequence of iid random variables on $\R^d$ with distribution $\nu(\cdot)/\lambda$; moreover, we assume that the $\xi_i$'s are independent of $N_t$. As usual we use the convention that $\sum_{i=1}^0\xi_i=0$.

The transition semigroup for a compound Poisson process is explicitly known. This allows us to reduce the coupling problem for a compound Poisson processes to that of a random walk. Let $P_t$ and $L$ be the semigroup and the generator for $L_t$, respectively. Then, it is well known that for any $f\in B_b(\R^d)$,
\begin{align*}
    Lf(x)
    &=\int \big(f(x+z)-f(x)\big)\nu(dz)\\
    &=\lambda\int \big(f(x+z)-f(x)\big)\nu_0(dz)
\end{align*}
and
\begin{equation}\label{coup0}
    {P}_t=e^{tL}=\sum_{n=0}^\infty \frac{t^nL^n}{n!}=e^{-\lambda t}\sum_{n=0}^\infty\frac{(t\lambda)^n{\nu_0^*}^n}{n!},\qquad t\ge0;
\end{equation}
here ${\nu_0^*}^n$ is the $n$-fold convolution of $\nu_0$ and ${\nu_0^*}^0:=\delta_0$.

\medskip

The following result explains the relationship of transition probabilities of compound Poisson processes and of random walks.
\begin{proposition}\label{couplingo}
    Let $P_t(x,\cdot)$ be the transition probability of compound Poisson process $L=(L_t)_{t\geq 0}$ and let $S=(S_n)_{n\geq 1}$, $S_n = \xi_1 + \ldots + \xi_n$, be a random walk where $(\xi_i)_{i\geq 1}$ are iid random variables with $\xi_1\sim \nu_0:=\nu/\nu(\R^d)$. Then, for all $x, y\in\R^d$
    \begin{align*}
        \|P_t(x,&\cdot)-P_t(y,\cdot)\|_{\var}\\
        &\le e^{-\lambda t}\sum_{n=0}^\infty \frac{(\lambda t)^n \, \|\Pp(x+S_n\in \cdot)-\Pp(y+S_n\in \cdot)\|_{\var}}{n!}\\
        &= e^{-\lambda t}\bigg[2(1-\delta_{x,y}) + \sum_{n=1}^\infty \frac{(\lambda t)^n\,\|\Pp(x+S_n\in \cdot)-\Pp(y+S_n\in\cdot)\|_{\var}}{n!}\bigg],
    \end{align*}
    where $\delta_{x,y}$ is a Kronecker delta function, i.e.\ $\delta_{x,y}=1$ if $x=y$, and $0$ otherwise.
\end{proposition}
\begin{proof}
    Let $P_t$ and $P^S_n$ be the semigroups of the compound Poisson process $L$ and the random walk $S$, respectively.
    Because of \eqref{coup0} we find
    \begin{align*}
        \|P_t(x,\cdot)-P_t(y,\cdot)\|_{\var}
        &=\sup_{\|f\|_\infty\le 1}|P_tf(x)-P_tf(y)|\\
        &=e^{-\lambda t}\bigg|\sup_{\|f\|_\infty\le 1}\sum_{n=0}^\infty\frac{(\lambda t)^n\big(\delta_x*{\nu_0^*}^n(f)-\delta_y*{\nu_0^*}^n(f)\big)}{n!}\bigg|\\
        &\leq e^{-\lambda t}\bigg|\sup_{\|f\|_\infty\le 1}\sum_{n=0}^\infty\frac{(\lambda t)^n\big(P^S_nf(x)-P^S_nf(y)\big)}{n!}\bigg|\\
        &\leq e^{-\lambda t}\sum_{n=0}^\infty\frac{(\lambda t)^n\sup_{\|f\|_\infty\le 1}|P^S_nf(x)-P^S_nf(y)|}{n!}\\
        &= e^{-\lambda t}\sum_{n=0}^\infty \frac{(\lambda t)^n\,\|\Pp(x+S_n\in \cdot)-\Pp(y+S_n\in \cdot)\|_{\var}}{n!},
    \end{align*}
    which proves the first assertion.

    For $n=0$ we have $S_0=0$; thus
    $$
        \|\Pp(x+S_0\in \cdot)-\Pp(y+S_0\in\cdot)\|_{\var}
        =\|\delta_x-\delta_y\|_{\var}
        =2(1-\delta_{x,y})
    $$
    and the second assertion follows.
\end{proof}

An immediate of Proposition \ref{couplingo} is the following estimate for $\|P_t(x,\cdot)-P_t(y,\cdot)\|_{\var}$ which is based on a similar inequality for $\|\Pp(x+S_n\in \cdot)-\Pp(y+S_n\in\cdot)\|_{\var}$.
\begin{proposition}\label{couplingo2}
    Assume that for all $x, y\in\R^d$ there is a constant $C(x,y)>0$ such that
    \begin{equation}\label{th120}
        \|\Pp(x+S_n\in \cdot)-\Pp(y+S_n\in \cdot)\|_{\var}\le \frac{C(x,y)}{\sqrt{n}}
        \qquad\text{\ for\ }n\ge1.
    \end{equation}
    Then,
    \begin{equation*}\label{th12}
        \|P_t(x,\cdot) - P_t(y,\cdot)\|_{\var}
        \leq 2e^{-\lambda t}(1-\delta_{x,y}) +\frac{\sqrt{2}C(x,y)(1-e^{-\lambda t})}{\sqrt{\lambda t}}.
    \end{equation*}
\end{proposition}
\begin{proof}
A combination of Proposition \ref{couplingo} and \eqref{th120} yields for all $x,y\in\R^d$ $$
    \|P_t(x,\cdot)-P_t(y,\cdot)\|_{\var}\le e^{-\lambda t}\bigg[2(1-\delta_{x,y})  +C(x,y)\sum_{n=1}^\infty \frac{(\lambda t)^n }{\sqrt{n}n!}\bigg].
$$
Jensen's inequality for the concave function $x^{1/2}$ gives
\begin{align*}
    \sum_{n=1}^\infty \sqrt{\frac 1n} \frac{(\lambda t)^n }{n!}
    &\le \frac{e^{\lambda t}-1}{(e^{\lambda t}-1)^{1/2}}\left(\sum_{n=1}^\infty\frac{(\lambda t)^n}{n\cdot n!}\right)^{1/2}\\
    &=\left(\frac{e^{\lambda t}-1}{\lambda t}\right)^{1/2}\left(\sum_{n=1}^\infty\frac{(\lambda t)^{n+1}}{(n+1)!}\cdot\frac{n+1}{n}\right)^{1/2}\\
    &\le\left(\frac{e^{\lambda t}-1}{\lambda t}\right)^{1/2} \sqrt 2 \Big(e^{\lambda t}-1-\lambda t\Big)^{1/2}\\
    &\le \frac{\sqrt{2}(e^{\lambda t}-1)}{\sqrt{\lambda t}}.
\end{align*}
The required assertion follows form the estimates above.
\end{proof}

We will now show that $L$ has the coupling property whenever $S$ has.
\begin{proposition}\label{wcp}
    Let $(L_t)_{t\geq 0}$ be the compound Poisson process with L\'{e}vy measure $\nu$, and
    $(S_n)_{n\geq 0}$, $S_n = S_0 + \xi_1 + \ldots + \xi_n$, be a random walk where $(\xi_i)_{i\geq 1}$ are iid random variables with $\xi_1\sim \nu_0:=\nu/\nu(\R^d)$. If $S_n$ admits a successful coupling, so does $L_t$.
\end{proposition}

\begin{proof}
    For $x, y\in\R^d$, let $L_t$ be a compound Poisson process starting from $x\in\R^d$. Then $L_t=\sum_{i=0}^{N_t}\xi_i$, where $\xi_0=x$ and where $(\xi_i)_{i\geq 1}$ are iid random variables with common distribution $\nu_0:=\nu/\nu(\R^d)$; $(N_t)_{t\geq 0}$ is a Poisson process with rate $\lambda:=\nu(\R^d)$. Moreover, $(N_t)_{t\geq 0}$ and $(\xi_i)_{i\geq 1}$ are independent.

    Set $S_0=x$ and $S_n=\sum_{i=0}^n \xi_i$ for $n\ge1$. Since $S$ has the coupling property, there exists another random walk $S'_n = \sum_{i=0}^n \xi_i'$ such that $S - S_0$ and $S'-S'_0$ have the same law and such that for any starting point $\xi_0' = y$ of $S'$ the coupling time
    $$
        T_{x,y}^S:=\inf\{k\ge 1\::\:S_k=S_k'\}
        \quad\text{is a.s.\ finite}.
    $$
    Without loss of generality, we can assume that $S_k=S_k'$ for $k\ge T_{x,y}^S$, and that $(\xi_i')_{i\geq 1}$ is independent of $(N_t)_{t\geq 0}$. Define
    $$
        L_t'=\sum_{i=0}^{N_t}\xi_i',\qquad t\ge0.
    $$
    Then $L' = (L_t')_{t\geq 0}$ is also a compound Poisson process with L\'{e}vy measure
    $\nu$ and starting point $L'_0=y$. In order to show that $L$ has the coupling property, it is enough to verify that
    \begin{equation}\label{com2}
        T_{x,y}^L:=\inf\{t>0\::\:L_t=L_t'\}<\infty.
    \end{equation}
    We claim that
    \begin{equation}\label{com3}
        T_{x,y}^L= K_{x,y}
        \quad\text{with}\quad K_{x,y}:=\inf\{t>0\::\:N_t\ge T_{x,y}^S\}.
    \end{equation}
    This implies \eqref{com2}. By assumption we know that for almost every $\omega$, $T_{x,y}^S(\omega)<\infty$. Since the Poisson process $N_t$ tends to infinity as $t\to\infty$, there exists $\tau_0(\omega)<\infty$ such that $N_{t}\ge T_{x,y}^S(\omega)$ for all $t\geq\tau_0(\omega)$. Therefore, \eqref{com3} tells us that $T_{x,y}^L\leq \tau_0 <\infty$.

    Let us finally prove \eqref{com3}. For this argument we assume that $\omega$ is fixed. Let $t>0$ be such that $N_t\ge T_{x,y}^S$, i.e.\ $t\geq K_{x,y}$. Since $S_k=S_k'$ for $k\ge T_{x,y}^S$, $S_{N_t}=S'_{N_t}$ and, by construction, $L_t=L_t'$; thus, $T_{x,y}^L\le t$ and since $t\geq K_{x,y}$ was arbitrary, we have $T^L_{x,y}\le K_{x,y}$. On the other hand, assume that $K_{x,y}>0$. Then, by the very definition of $K_{x,y}$, for any $\varepsilon>0$, there exists $t_\varepsilon>0$ such that $t_\varepsilon>K_{x,y}-\varepsilon$ and $N_{t_\varepsilon}\le T_{x,y}^S-1$. Hence, $S_{N_{t_\varepsilon}}\neq S'_{N_{t_\varepsilon}}$, i.e.\ $L_{t_\varepsilon}\neq L'_{t_\varepsilon}$. Therefore, $T_{x,y}^L\ge t_\varepsilon>K_{x,y}-\varepsilon$. Letting
    $\varepsilon\rightarrow0$, we get $T_{x,y}^L\ge K_{x,y}$ and the proof is complete.
\end{proof}

Note that the proof of Proposition \ref{wcp} already gives an estimate for the rate at which coupling occurs: for all $t>0$
$$
    \Pp(T_{x,y}^L>t)
    =\Pp(N_t<T_{x,y}^S)
    =e^{-\lambda t}\bigg[1+\sum_{k=1}^\infty \Pp(T_{x,y}^S>k)\frac{(\lambda t)^k}{k!}\bigg].
$$
What remains to be done is to get estimates for the coupling time of
the random walk $S$, $\Pp(T_{x,y}^S>k), k\ge1$. This requires
concrete coupling constructions for $S$; the most interesting random
walk couplings rely on a suitable coupling of the steps $\xi_j$ and
$\xi'_j$ of $S$ and $S'$, respectively. In the next section we will,
therefore, consider the Mineka and Lindvall-Rogers couplings.

\medskip
We close this section with two comments on  Proposition \ref{wcp}.
\begin{remark}
    (1) Theorem \ref{ex2} below will show, that the converse of Proposition \ref{wcp} is also true: \emph{Let $L=(L_t)_{t\geq 0}$ be a compound Poisson process with L\'{e}vy measure $\nu$, and $S=(S_n)_{n\geq 0}$, $S_n = S_0 + \xi_1 + \ldots + \xi_n$, be a random walk where $(\xi_i)_{i\geq 1}$ are iid random variables with $\xi_1\sim \nu_0:=\nu/\nu(\R^d)$. If $L$ has the coupling property so does $S$.}

    (2) Proposition \ref{wcp} can be easily generalized to more general settings, see also \cite{BSW}. More precisely, let $(X_t)_{t\geq 0}$ be a Markov process on $\R^d$ and let $(S_t)_{t\geq 0}$ be a subordinator (i.e.\ an increasing L\'evy process) which is independent of $X_t$. If $X_t$ has the coupling property and if $S_t$ tends to infinity as $t\to\infty$, then the subordinate process $X_{S_t}$ also has the coupling property.
\end{remark}

\section{The Mineka and Lindvall-Rogers Couplings --- A Review}\label{section3}

Let $S=(S_n)_{n\geq 1}$, $S_n=\xi_1+\ldots+\xi_n$ be a random walk on $\R^d$ with iid steps $(\xi_i)_{i\geq 0}$ such that $\xi_1\sim\nu_0$. The main result of this section is

\begin{theorem}\label{rw}
Suppose that for some $\delta>0$,
\begin{equation}\label{rw1}
    \eta_0
    =\eta_0(\delta)
    :=\inf_{x\in\R^d,|x|\le \delta}\nu_0\wedge (\delta_x*\nu_0)(\R^d)>0.
\end{equation}
Then there exists a constant $C:=C(\delta,\eta_0)>0$ such that for all $x, y\in\R^d$ and $n\ge1$,
\begin{equation}\label{rw2}
    \|\Pp(x+S_n\in \cdot)-\Pp(y+S_n\in \cdot)\|_{\var}
    \le \frac{C(1+|x-y|)}{\sqrt{n}}.
\end{equation}
\end{theorem}

The proof of Theorem \ref{rw} is mainly based on Mineka's coupling
\cite{M}, see also \cite[Chapter II, Section 14, Page 44---Page
47]{Li}, and the coupling argument of the zero-two law for random
walks proved in \cite[Proposition~1]{LR} by Lindvall and Rogers.
These papers do not contain an estimate as explicit as \eqref{rw2}.
Therefore we decided to include a detailed proof on our own which
again highlights the role of the sufficient condition \eqref{rw1}.

\medskip

We begin with an auxiliary result which describes the total variation norm of a signed measure under a non-degenerate linear transformation.
\begin{lemma}\label{rm1}
    Let $\mu$ be a probability measure $\mu$ on $\Rd$. Then we have for all $x, y\in\R^d$
    $$
        \|\delta_x\ast \mu-\delta_y\ast \mu\|_{\var}
        =\|\delta_{x-y}\ast \mu- \mu\|_{\var}
        =\|\delta_{|x-y|e_1}\ast (\mu\circ R_{x-y}^{-1})-\mu\circ R_{x-y}^{-1}\|_{\var},
    $$
    where $e_1=(1,0,\cdots,0)\in \Rd$ and $R_a$ is a non-degenerate rotation such that $R_a a=|a|e_1$. In particular, for any $a\in\R^d$,
    $$
        \|\delta_{a}\ast \mu- \mu\|_{\var}
        =\|\delta_{-a}\ast \mu-\mu\|_{\var}.
    $$
\end{lemma}

\begin{proof}
    Using the definition of the total variation norm we get
    \begin{align*}
        \|\delta_x\ast \mu-\delta_y\ast \mu\|_{\var}
        &= 2\sup_{A\in \Bb(\R^d)}\big|\delta_x*\mu(A)-\delta_y*\mu(A)|\\
        &= 2\sup_{A\in \Bb(\R^d)}|\mu(A-x)-\mu(A-y)|\\
        &= 2\sup_{B\in \Bb(\R^d)}|\mu(B-(x-y))-\mu(B)|\\
        &= \|\delta_{x-y}\ast \mu- \mu\|_{\var}.
    \end{align*}
    Now let $a\in\R^d$ and denote by $R_a$ the rotation such that $R_a a = |a|e_1$. Clearly,
    $$
        \mu \circ R_{a} (A)
        =\mu\{R_a x\in\R^d\::\: x\in A\}
        =\mu\{y\in\R^d: R_a^{-1}(y)\in A\},
        \quad A\in\mathscr B(\Rd).
    $$
    Then,
    \begin{align*}
        \|\delta_a\ast \mu- \mu\|_{\var}
        &= 2\sup_{A\in \Bb(\R^d)}\big|\mu(A-a)-\mu(A)|\\
        &= 2\sup_{A\in \Bb(\R^d)}\big|\mu\circ R_a^{-1}(R_a(A-a))-\mu\circ R_a^{-1}(R_aA)\big|\\
        &= 2\sup_{A\in \Bb(\R^d)}|\mu\circ R_a^{-1}(R_aA-|a|e_1)-\mu\circ R_a^{-1}(R_aA)|\\
        &= 2\sup_{B\in \Bb(\R^d)}|\mu\circ R_a^{-1}(B-|a|e_1)-\mu\circ R_a^{-1}(B)|\\
        &= \|\delta_{|a|e_1}\ast(\mu \circ R^{-1}_{a})-\mu \circ R^{-1}_{a}\big\|_{\var}.
    \qedhere
    \end{align*}
\end{proof}

\begin{proposition}\label{time}
Under \eqref{rw1}, there exists a constant $C=C(\eta_0)>0$ such that
\begin{equation}\label{rw3}
    \sup_{|x-y|\le \delta} \|\Pp(x+S_n\in \cdot)-\Pp(y+S_n\in \cdot)\|_{\var}
    \le \frac{C}{\sqrt{n}}.
\end{equation}
\end{proposition}
\begin{proof}
\emph{Step 1.}
It is easy to see that for any $a\in\R^d$ and any probability measure $\mu$,
$$
    \mu^{*n}\circ R_a^{-1}=(\mu \circ R_a^{-1})^{*n}.
$$
Lemma \ref{rm1} shows that \eqref{rw3} is equivalent to the following estimate
\begin{equation}\label{rw3a}
    \sup_{|a|\le \delta} \|\Pp(|a|e_1+S_{a,n}\in \cdot)-\Pp(S_{a,n}\in \cdot)\|_{\var}
    \le \frac{C}{\sqrt{n}}.
\end{equation}
Here, $S_{a,n}$ is a random walk in $\R^d$ with iid steps $\xi_{a,1}, \xi_{a,2}, \ldots$ and $\xi_{a,1} \sim \nu_0\circ R_a^{-1}$.

On the other hand, Lemma \ref{rm1} also shows that
\begin{align*}
    (\nu_0\circ R_a^{-1}&)\wedge (\delta_{|a|e_1}*(\nu_0\circ R_a^{-1}))(\R^d)\\
    &=1-\frac 12\,\|\nu_0\circ R_a^{-1}-(\delta_{|a|e_1}*(\nu_0\circ R_a^{-1}))\|_{\var}\\
    &=1-\frac 12\,\|\nu_0-(\delta_{a}*\nu_0)\|_{\var}\\
    &=\nu_0\wedge(\delta_{a}*\nu_0)(\R^d).
\end{align*}
Therefore, \eqref{rw1} implies that for any $a\in\R^d$
\begin{equation}\label{rw4}
    \inf_{|a|\le \delta}\left\{(\nu_0\circ R_a^{-1})\wedge (\delta_{|a|e_1}*(\nu_0\circ R_a^{-1}))(\R^d)\right\}
    = \inf_{|a|\le \delta}\left\{\nu_0\wedge (\delta_a*\nu_0)(\R^d)\right\}
    >0.
\end{equation}
In order to simplify the notation, we use $\nu:=\nu_0\circ R_a^{-1}$ and $S_n:=S_{a,n}$. With this notation \eqref{rw3a} becomes
\begin{equation}\label{rw51}
    \sup_{|a|\le \delta} \|\Pp(|a| e_1+S_n\in \cdot)-\Pp(S_n\in \cdot)\|_{\var}
    \le \frac{C}{\sqrt{n}}.
\end{equation}

\medskip
\emph{Step 2.} Assume that $|a|\in (0,\delta]$ and set $\nu_{|a|}=\delta_{|a| e_1}*\nu$ and $\nu_{-|a|}=\delta_{-|a| e_1}*\nu$. Let $(\xi,\Delta\xi)\in \Rd \times \Rd$ be a pair of random variables with the following distribution
$$
    \Pp\big((\xi,\Delta\xi)\in A\times D\big)
    =
    \begin{cases}
    \qquad \frac 12 (\nu\wedge\nu_{-|a|})(A), & D=\{|a| e_1\};\\
    \qquad \frac 12 (\nu\wedge\nu_{|a|})(A),       & D= \{-|a| e_1\};\\
    \big(\nu- \frac 12 (\nu\wedge\nu_{-|a|}+\nu\wedge\nu_{|a|})\big)(A), & D=\{0\};
    \end{cases}
$$
where $A\in\Bb(\R^d)$ and $D\in\big\{ \{-|a| e_1\}, \{0\}, \{|a| e_1\}\big\}$. We see from \eqref{rw4} that
$$
    \Pp(\Delta\xi=|a| e_1)
    = \frac{1}{2}\big(\nu\wedge\big(\delta_{-|a| e_1}*\nu)\big)(\R^d)
    \ge \frac{1}{2}\inf_{|a|\le \delta}\nu_0\wedge (\delta_a*\nu_0)(\R^d).
$$
By Lemma \ref{rm1},
\begin{align*}
    \Pp(\Delta\xi=-|a| e_1)
    &=\frac{1}{2}\big(\nu\wedge\big(\delta_{|a| e_1}*\nu)\big)(\R^d)\\
    &=\frac{1}{2}\big(\nu\wedge\big(\delta_{-|a| e_1}*\nu)\big)(\R^d)\\
    &=\Pp(\Delta\xi=|a| e_1).
\end{align*}

It is clear that the distribution of $\xi$ is $\nu$. Let $\xi'=\xi+\Delta\xi$. We claim that the distribution of $\xi'$ is also $\nu$. Indeed, for any $A\in\mathscr{B}(\R^d)$,
\begin{align*}
    \Pp(\xi'\in A)
    &=\Pp(\xi-|a| e_1\in A, \Delta\xi=-|a| e_1)\\
    &\qquad + \Pp(\xi+|a| e_1\in A, \Delta\xi=|a| e_1)+\Pp(\xi \in A, \Delta\xi=0)\\
    &= \frac{\delta_{-|a| e_1}*(\nu\wedge\nu_{|a|})(A)}2 + \frac{\delta_{|a| e_1}*(\nu\wedge\nu_{-|a|})(A)}2\\
    &\qquad  + \left(\nu- \frac{\nu\wedge\nu_{-|a|}+\nu\wedge\nu_{|a|}}2\right)(A)\\
    &=\mu(A),
\end{align*}
where we have used that $\delta_{-|a|
e_1}*(\nu\wedge\nu_{|a|})=\nu\wedge\nu_{-|a|}$ and $\delta_{|a|
e_1}*(\nu\wedge\nu_{-|a|})=\nu\wedge\nu_{|a|}$. Now we construct the
coupling $(S_n,S_n')$ of $S_n$ with the iid pairs $(\xi_i,\xi'_i),
i\ge1$ where $(\xi_1,\xi_1') \sim (\xi,\xi')$. Since
$\xi'_i-\xi_i=\Delta\xi$, we know that $\xi_i-\xi'_i$ is, for all
$i\geq 1$, symmetrically distributed, takes only the values $-|a|
e_1$, $0$ and $|a| e_1$. Because of \eqref{rw4}, we have

\begin{equation*}\label{rw5}\begin{aligned}
    P(a)
    :=\Pp(\xi^{1\,\prime}_i-\xi^1_{i}=0)
    &= \big(\nu-\tfrac 12\,(\nu\wedge\nu_{-|a|}+\nu\wedge\nu_{|a|})\big)(\R^d)\\
    &= 1-\nu\wedge\nu_{-|a|}(\R^d)\\
    &\leq 1-\inf_{|a|\le \delta}\nu_0\wedge (\delta_a*\nu_0)(\R^d)\\
    &=: \gamma(\delta)<1,
\end{aligned}\end{equation*}
where $\xi_i=(\xi^1_{i},\xi^2_{i},\cdots,\xi^d_{i})$ and
$\xi'_i=(\xi^{1\,\prime}_{i},\xi^{2\,\prime}_{i},\cdots,\xi^{d\,\prime}_{i})$.
Set $S^j_{n}=\sum_{i=1}^n\xi^j_{i}$ and
$S^{j\,\prime}_{n}=\sum_{i=1}^n\xi^{j\,\prime}_{i}$. We observe that
$(S^1-S^{1\,\prime})$ is a random walk, whose step sizes are $-|a|$,
$0$ and $|a|$ with  probability $\frac 12(1-P(a))$, $P(a)$ and
$\frac 12(1-P(a))$, respectively. Since $S^j_{n}=S^{j\,\prime}_{n}$
for $2\le j\le d$, we get
\begin{equation}\label{rw6}
    \|\Pp(|a| e_1+S_n\in \cdot)-\Pp(S_n\in \cdot)\|_{\var}
    \le 2\Pp(T^S>n),
\end{equation}
where
$$
    T^S=\inf\{k\ge1\::\: S^1_{k}=S^{1\,\prime}_{k}+|a|\}.
$$

\medskip
\emph{Step 3.}
We will now estimate $\Pp(T^S>n)$. Let $V_1, V_2, \ldots$ be iid symmetric random variables, whose common distribution is given by
$$
    \Pp(V_i=x)
    = \begin{cases}
        \frac 12(1-P(a)),     &x=-|a|;\\
        \frac 12(1-P(a)),     &x=|a|;\\
        \qquad P(a),    &x=0.
    \end{cases}
$$
Define $Z_n=\sum_{i=1}^n V_i$. We have seen in Step 2 that
$T^S=\inf\{n\ge 1\::\: Z_n=|a|\}$. Then, by the reflection
principle,
\begin{align*}
    \Pp(T^S>n)
    =\Pp\left(\max_{k\leq n} Z_k < |a|\right)
    \leq 2\,\Pp\left(0 \leq Z_n \leq |a|\right).
\end{align*}
Since $Z$ is the sum of iid random variables with mean $0$ and variance $\sigma^2 = |a|^2 (1-P(a))$, we can use the central limit theorem to deduce, for sufficiently large values of $n$,
\begin{align*}
    \Pp(T^S>n)
    &= 2\,\Pp\left(0\leq \frac{Z_n}{|a|\sqrt{1-P(a)}\sqrt n} \leq \frac 1{\sqrt{1-P(a)}\sqrt n}\right)\\
    &\leq 2\,\Pp\left(0\leq \frac{Z_n}{|a|\sqrt{1-P(a)}\sqrt n} \leq \frac 1{\sqrt{1-\gamma(\delta)}\sqrt n}\right)\\
    &\leq \frac{C}{\sqrt{2\pi}} \int_{0}^{1/\sqrt{1-\gamma(\delta)}\sqrt n} e^{-x^2/2}\,dx\\
    &\leq \frac{C_{1,\gamma(\delta)}}{\sqrt n}.
\end{align*}
In the first inequality above we have used the fact that
$1-P(a)\ge 1-\gamma(\delta)>0.$
Therefore we find from
\eqref{rw6} for all large $n\geq 1$
$$
    \|\Pp(|a| e_1+S_n\in \cdot)-\Pp(S_n\in\cdot)\|_{\var}
    \le\frac{2C_{2, \gamma(\delta)}}{\sqrt{n}}.
$$
Since the right-hand side is bounded by $2$, this estimate actually
holds for all $n\geq 1$.

We can now use Lemma \ref{rm1} to get
$$
    \|\Pp(\pm|a| e_1+S_n\in \cdot)-\Pp(S_n\in\cdot)\|_{\var}
    \le\frac{2C_{2, \gamma(\delta)}}{\sqrt{n}}
$$
which immediately yields \eqref{rw51}, since $|a|\leq\delta$ was arbitrary.
\end{proof}

We close this section with the proofs of Theorems \ref{rw} and \ref{th2}.

\begin{proof}[Proof of Theorem $\ref{rw}$]
For any $x, y\in\R^d$, set $k=\big[\frac{|x-y|}{\delta}\big]+1$.
Pick $x_0, x_1, \ldots, x_k\in\Rd$ such that $x_0=x$, $x_k=y$ and
$|x_i-x_{i-1}|\le \delta$ for $1\le i\le k$. By Proposition
\ref{time},
\begin{align*}
    \|\Pp(x+S_n\in\cdot)-\Pp(y+S_n\in \cdot)\|_{\var}
    &\le \sum_{i=1}^k\|\Pp(x_i+S_n\in\cdot)-\Pp(x_{i-1}+S_n\in \cdot)\|_{\var}\\
    &\le \frac{C(\delta,\eta_0)(1+|x-y|)}{\sqrt{n}},
\end{align*}
which is what we claimed.
\end{proof}

\begin{proof}[Proof of Theorem $\ref{th2}$]
\emph{Step 1.} Assume first that the L\'evy triplet is of the form
$(0,0,\nu)$ and that the L\'evy measure satisfies $\lambda =
\nu(\Rd)<\infty$. This means that $X_t$ is compound Poisson process.
We use the notations from Section \ref{section2}. For all
$x\in\R^d$,
$$
    \nu\wedge (\delta_x*\nu)(\R^d)
    =\lambda\left[\nu_0\wedge (\delta_x*\nu_0)(\R^d)\right] > 0.
$$
Therefore, we can apply Proposition \ref{couplingo2} and Theorem \ref{rw} to get Theorem \ref{th2} in this case.

\medskip
\emph{Step 2.} If $(X_t)_{t\geq 0}$ is a general L\'{e}vy  process with L\'evy triplet $(b,Q,\nu)$, we split $X_t$ into two independent parts
$$
    X_t=X'_t+X''_t,
$$
where $X'_t$ is the compound Poisson process with L\'{e}vy triplet $(0,0,\nu_\varepsilon)$---$\nu_\varepsilon$ is defined in \eqref{eq:th2}---and $X''_t$ is the L\'evy process with triplet $(b,Q,\nu-\nu_\epsilon)$. Denote by $P'_t, P''_t$, $P'_t(x,dy), P''_t(x,dy)$ the transition semigroups and transition functions of the processes $X'_t$ and $X''_t$, respectively. Then, $P_t=P'_t P''_t$. Observe that $P''_t$ is a contraction semigroup, i.e.\ $\|P''_tf\|_\infty \leq 1$ whenever $\|f\|_\infty\leq 1$. Therefore,
\begin{align*}
    \|P_t(x,\cdot)-P_t(y,\cdot)\|_{\var}
    &= \sup_{\|f\|_\infty\le 1}\big|P_tf(x)-P_tf(y)\big|\\
    &=\sup_{\|f\|_\infty\le 1} \big|P_t'P_t''f(x)-P_t' P_t''f(y)\big|\\
    &\le \sup_{\|h\|_\infty\le 1} \big|P'_th(x)-P'_th(y)\big|\\
    &= \|P'_t(x,\cdot)-P'_t(y,\cdot)\|_{\var},
\end{align*}
which reduces the general case to the compound Poisson setting considered in the first part.
\end{proof}

\section{Extensions: The Lindvall-Rogers `Zero-Two' Law}\label{section4}
Motivated by Lindvall-Rogers's zero-two law for random walks
\cite[Propsotion 1]{LR}, we present a necessary and sufficient
condition for the coupling property of a L\'{e}vy process. We will
add a few simple sufficient criteria in terms of the L\'{e}vy
measure which are easy to verify. Throughout this section we assume
that $(X_t)_{t\geq 0}$ is a $d$-dimensional L\'evy process with
L\'evy measure $\nu\not\equiv 0$; as usual, $X_0=0$. By
$P_t(x,\cdot)$ and $P_t$ we denote the transition probability and
transition semigroup, respectively.

\begin{theorem}[Criterion for successful couplings]\label{ex1}
The following statements are equivalent:
\begin{itemize}
\item[(1)] The L\'{e}vy process $(X_t)_{t\geq 0}$ has the coupling property.
\item[(2)] There exists $t_0> 0$ such that for any $t\ge t_0$, the transition probability $P_t(x,\cdot)$ has (with respect to Lebesgue measure) an absolutely continuous component.
\end{itemize}
In either case, for every $x, y\in\R^d$, there exists a constant $C(x,y)>0$ such that
\begin{equation}\label{ex11}
    \|P_t(x,\cdot)-P_t(y,\cdot)\|_{\var}
    \le \frac{C(x,y)}{\sqrt{t}},\qquad t>0.
\end{equation}
\end{theorem}

If the L\'evy process has the strong Feller property, i.e.\ the corresponding semigroup maps $B_b(\Rd)$ into $C_b(\Rd)$, Theorem \ref{ex1} becomes particularly simple.
\begin{corollary}
    Suppose that there exists some $t_0>0$ such that the semigroup $P_{t_0}$ maps $B_b(\R^d)$ into $C_b(\R^d)$. Then, the L\'{e}vy process $X_t$ has the coupling property. In particular, every L\'{e}vy process which enjoys the strong Feller property has the coupling property.
\end{corollary}
\begin{proof}
    By assumption $P_{t_0}$ is a convolution operator which maps $B_b(\R^d)$ into $C_b(\R^d)$. Due to a result by Hawkes, cf.\ \cite{HAW} or \cite[Lemma 4.8.20]{jac}, $P_{t_0}$ and all $P_t$ with $t\geq t_0$ are of the form $P_{t}(x) = p_{t}*f(x)$, where $p_t(x)$ is the transition density of the process. Therefore condition (2) of Theorem \ref{ex1} is satisfied.
\end{proof}

Now, we turn to the proof of Theorem \ref{ex1}.
\begin{proof}[Proof of Theorem $\ref{ex1}$]
As mentioned in Section \ref{section1}, the coupling property is equivalent to \eqref{prex1}. Observe that
\begin{align*}
    \|\mu_1P_t-\mu_2P_t\|_{\var}
    &\le \|\mu_1P_t-P_t\|_{\var}+\|\mu_2P_t-P_t\|_{\var}\\
    &\le \int  \|\delta_x*P_t-\delta_0*P_t\|_{_{\var}}\,\mu_1(dx)\\
    &\qquad+\int  \|\delta_x*P_t-\delta_0*P_t\|_{_{\var}}\,\mu_2(dx).
\end{align*}
Thus, if
\begin{equation}\label{prex2}
    \lim_{t\rightarrow\infty}\|\delta_x*P_t-\delta_0*P_t\|_{\var}=0
    \qquad\text{for\ \ }x\in\R^d,
\end{equation}
then we can use the dominated convergence theorem to see that
\eqref{prex1} holds. Therefore, the assertions \eqref{prex1} and
\eqref{prex2} are equivalent.

Since $\|\delta_x*P_t-\delta_0*P_t\|_{\var}$ is decreasing in $t$, we see that \eqref{prex2} is also equivalent to
\begin{equation}\label{prex3}
    \lim_{n\rightarrow\infty}\|\delta_x*P_n-\delta_0*P_n\|_{\var}=0
    \qquad\text{for\ \ }x\in\R^d.
\end{equation}
Let $\xi_1, \xi_2, \ldots$ be iid random variables on $\R^d$ with
$\xi_1\sim\mu_1:=\Pp(X_1\in \cdot)$ and set $S_n=\sum_{i=1}^n\xi_i$
for $n\ge1$ and $S_0=0$. Since the increments of a L\'evy process
are independent and stationary, \eqref{prex3} is the same as
\begin{equation}\label{prex4}
    \lim_{n\rightarrow\infty}\|\Pp(x+S_n\in \cdot)-\Pp(S_n\in \cdot)\|_{\var}=0
    \qquad\text{for\ \ }x\in\R^d.
\end{equation}
According to \cite[Remark (ii), Page 124---Page 125]{LR}
or
\cite[Chapter 3, Section 3, Theorem 3.9]{RU},
 \eqref{prex4}
holds if, and only if, $\mu_1$ is spread out, i.e.\ for some
$m\ge1$, $\mu_1^{*m}=P_m(0,\cdot):=\Pp(X_m\in \cdot)$ has an
absolutely continuous component. Since the semigroup of a L\'{e}vy
process is a convolution semigroup, it is easy to see that for every
$t\ge m$, the transition probability $P_t(x,\cdot)$ has an
absolutely continuous part. Combining all the assertions above, we
have proved that the statements (1) and (2) are equivalent.

Moreover, the arguments used in the proof of Proposition \ref{time} together with \cite[Proposition 1]{LR} show that
$$
    \|\Pp(x+S_n\in \cdot)-\Pp(S_n\in\cdot)\|_{\var}=\mathsf{O}(n^{-1/2})\qquad\textrm{as\ \ }n\rightarrow\infty
$$
whenever the random walk $S_n$ has the coupling property. Therefore, \eqref{ex11} follows from the arguments used in the first part of the proof, in particular since $t\mapsto \|\delta_x*P_t-\delta_0*P_t\|_{\var}$ is decreasing and $\|\Pp(x+S_n\in \cdot)-\Pp(S_n\in \cdot)\|_{\var} = \|\delta_x*P_n-\delta_0*P_n\|_{\var}$.
\end{proof}

Let us finally derive some sufficient conditions in terms of the L\'{e}vy measure, which extend Theorem \ref{th2} and \cite[Theorem~3.1]{W1}.
\begin{theorem}[Sufficient criteria for successful couplings]\label{ex2}
    Let $(X_t)_{t\geq 0}$ be a L\'{e}vy process on $\R^d$ with L\'{e}vy measure $\nu\not\equiv 0$ and define $\varepsilon>0$ as in \eqref{eq:th2}, i.e.\ for  $B\in\mathscr{B}(\R^d)$
    $$
        \nu_\varepsilon(B)
        =
        \begin{cases}\qquad\nu(B),\qquad &\nu(\R^d)<\infty;\\
        \nu\{z\in B\::\: |z|\ge \varepsilon\},& \nu(\R^d)=\infty.
        \end{cases}
    $$
    If there exists some $\varepsilon>0$ such that one of the following conditions is satisfied
    \begin{itemize}
    \item[(1)] For some $l\ge1$,  ${{\nu}_\varepsilon^*}{^l}$ has an absolutely continuous component,
    \item[(2)] There exist $\delta>0$ and $l\ge1$ such that $\displaystyle\inf_{x\in\R^d,|x|\le \delta}{\nu_\varepsilon^*}{^l} \wedge(\delta_x*{\nu_\varepsilon^*}{^l})(\R^d)>0$,
    \end{itemize}
    then the process $X_t$ has the coupling property.

    Conversely, assume that $\nu(\R^d)<\infty$ and $X_t$ is a compound Poisson process with L\'{e}vy measure $\nu$. If $X_t$ has the coupling property, then there is some $l\geq 0$ such that  $\nu^{*l}$ has an absolutely continuous component.
\end{theorem}
\begin{proof}
\emph{Step 1.}
The argument used in the proof of Theorem \ref{th2} shows that we only have to consider the coupling property for a compound Poisson process, whose L\'{e}vy measure is of the form $\nu_\varepsilon$. Let $S=(S_n)_{n\geq 1}$, $S_n = \xi_1+\ldots+\xi_n$, be a random walk on $\R^d$ with iid steps $\xi_1, \xi_2, \ldots$ such that $\xi_1\sim \nu_\varepsilon/\nu_\varepsilon(\R^d)$. Because of Proposition \ref{wcp}, it is sufficient to show that, under the assumptions stated in the theorem, $S$ has the coupling property.

As we have pointed out in the proof of Theorem \ref{ex1}, \cite[Remark (ii), Page 124---Page 125]{LR} shows that $S$ has the coupling property if, and only if, condition (1) holds. Again by \cite[Remark (ii), Page 124---Page 125]{LR}, condition (1) is
equivalent to saying that for any $x\in\R^d$, there exists $l\ge0$ such that ${\nu_\varepsilon^*}{^l}\wedge
(\delta_x*{\nu_\varepsilon^*}{^l})(\R^d)>0$. Clearly, such a condition is hard to check in applications.

Let $Z = (Z_n)_{n\geq 1}$, $Z_n=\zeta_1+\ldots+\zeta_n$, be a random
walk on $\R^d$ with iid steps $\zeta_1, \zeta_2, \ldots$ such that
$\zeta_1\sim \nu_\varepsilon^{*l}/\nu^l_\varepsilon(\R^d)$. That is,
$\zeta_i=\sum_{k=(i-1)l+1}^{il}\xi_k$, where $\xi_i$ is the step of
the random walk $S$ from the paragraph above. If (2) holds, Theorem
\ref{rw} shows that $Z$, hence $S$, has the coupling property.

\medskip
\emph{Step 2.} Let $(X_t)_{t\geq 0}$ be a compound Poisson process
with L\'{e}vy measure $\nu$ and suppose that $X_t$ has the coupling
property. Moreover, assume that none of the measures $\nu^{*l}$,
$l\geq 1$, has an absolutely continuous component. By the Lebesgue
decomposition, each measure $\nu^{*l}$ is mutually singular with
respect to Lebesgue measure $\textup{Leb}$. Thus, for every $l\ge1$,
there exists some set $A_l\in\Bb(\R^d)$ such that
$\textup{Leb}(A_l)= \nu^{*l}(A_l^c)=0$. For $A:=\bigcup_{i=1}^\infty
A_i$ we have $\textup{Leb}(A)=\nu^{*l}(A^c)=0$ for each $l\ge1$.
Therefore, by \eqref{coup0}, for every $x\in\R^d$ and $t>0$, the
transition probability $P_t(x,\cdot)$ of $X_t$ is singular with
respect to Lebesgue measure $\textup{Leb}$. This shows that
condition (2) of Theorem \ref{ex1} cannot hold, i.e.\ $X_t$ does not
have the coupling property. Since this contradicts our assumption,
the proof is finished.
\end{proof}

Theorem \ref{ex2} immediately yields that
\begin{corollary}
    Any L\'{e}vy process whose L\'{e}vy measure possesses an absolutely continuous component has the coupling property.
\end{corollary}

The coupling property of a L\'{e}vy process is intimately connected with the choice of state space. According to Theorem \ref{ex2}, a Poisson process on $\R$ does not have the coupling property, see also the discussion in Remark \ref{th22re}. If, however, the process is considered on $\Z$, the situation changes.

\begin{proposition}
    A Poisson process $X_t$ with state space $\Z$ has the coupling property.
\end{proposition}
\begin{proof}
    We use the coupling and shift coupling properties proved in \cite{CW}. Shift coupling is a slightly weaker notion than coupling. A Markov process $(X_t)_{t\geq 0}$ is said to have the shift coupling property, if for any two initial distributions $\mu, \nu$, there exists a coupling $(X_t,Y_t)$ with marginal processes such that
    \begin{itemize}
    \item $X_0\sim\mu$ and $Y_0\sim \nu$;
    \item there are finite stopping times $T_1, T_2$ such that $X_{T_1}=Y_{T_2}$.
    \end{itemize}
    Let $\lambda$ be the intensity of the Poisson process $X_t$. Then the infinitesimal generator is given by
    $Lf(i)=\lambda(f(i+1)-f(i))$ for $i\in \Z$. Thus, all harmonic functions $f:\Z \to \R$ are constant and, by
\cite[Theorem 1 and its second remark]{CG}
        or \cite[Theorem~2]{CW}, the process has the
shift coupling property.

    Similar to the proof of \cite[Proposition 3.3]{W1}, for any $s,t>0$, $i\in \Z$ and $f\in B_b(\Z)$
  with $f\ge0$,
    $$
        P_{t+s}(i)=\Ee f(i+X_{t+s})\ge\Ee f((i+X_{t}){\I_{\{X_{t+s}-X_t=0\}}})=e^{-\lambda s}\Ee f(i+X_{t})=e^{-\lambda s}P_tf(i),
    $$
    which shows that $X_t$ has the coupling property, cf.\ \cite[Theorem 5]{CW}.
\end{proof}

\medskip
\begin{acknowledgement} We would like to thank
Professor F.-Y. Wang for helpful discussions in Beijing and Swansea.
Financial support through DFG (grant Schi 419/5-1) (for R.L.S.) and
the Alexander-von-Humboldt Foundation and the Natural Science
Foundation of Fujian $($No.\ 2010J05002$)$ (for J.W.) is gratefully
acknowledged. We also would like to think an anonymous referee for
carefully reading the first draft of this paper.
\end{acknowledgement}

\end{document}